\documentclass{amsart}
\usepackage{amsmath,amsthm,amssymb,amsfonts}
\usepackage[nobysame,abbrev,alphabetic]{amsrefs}
\usepackage{color}
\usepackage[shortlabels]{enumitem}
\numberwithin{equation}{section}

\newcommand{\w}{\wedge}

\DeclareFontFamily{OT1}{rsfs}{}
\DeclareFontShape{OT1}{rsfs}{n}{it}{<-> rsfs10}{}
\DeclareMathAlphabet{\mathscr}{OT1}{rsfs}{n}{it}

\theoremstyle{definition}

\newcommand{\IGNORE}[1]{}

\newtheorem{theorem}{Theorem}[section]
\newtheorem{proposition}[theorem]{Proposition}

\newtheorem{lemma}[theorem]{Lemma}

\theoremstyle{definition}
\newtheorem{definition}[theorem]{Definition}
\newtheorem{remark}[theorem]{Remark}

\newcommand{\hide}[1]{}

\newcommand{\bR}{\mathbb{R}}
\newcommand{\lv}{\lvert}
\newcommand{\rv}{\rvert}
\newcommand{\lp}{\langle}
\newcommand{\rp}{\rangle}
\newcommand{\suchthat}{\mathrel{}\middle|\mathrel{}}  
\newcommand{\bN}{\mathbb{N}}

\newcommand{\mB}{\mathcal{B}}
\newcommand{\mC}{\mathcal{C}}
\newcommand{\mE}{\mathcal{E}}
\newcommand{\mQ}{\mathcal{Q}}
\newcommand{\odelta}{\overline{\delta}}
\newcommand{\onabla}{\overline{\nabla}}
\newcommand{\oDelta}{\overline{\Delta}}
\def\sideremark#1{\ifvmode\leavevmode\fi\vadjust{\vbox to0pt{\vss
 \hbox to 0pt{\hskip\hsize\hskip1em
 \vbox{\hsize3cm\tiny\raggedright\pretolerance10000
 \noindent #1\hfill}\hss}\vbox to8pt{\vfil}\vss}}}        

\DeclareMathOperator{\trace}{tr}
\DeclareMathOperator{\Tr}{Tr}

\newcommand{\oX}{\overline{X}}

\begin{document}

\title[Sobolev trace inequality]{A fully nonlinear Sobolev trace inequality}

\author[Jeffrey S.\ Case]{Jeffrey S.\ Case}
\address{Jeffrey S.\ Case \\ 109 McAllister Building \\ Penn State University \\ University Park, PA 16802}
\email{jscase@psu.edu}

\author[Yi Wang]{Yi Wang}
\address{Yi Wang, Department of Mathematics, Johns Hopkins University, 404 Krieger Hall, 3400 N. Charles Street, Baltimore, MD 21218,}
\email{ywang@math.jhu.edu}
\setcounter{page}{1}


\begin{abstract}
The $k$-Hessian operator $\sigma_k$ is the $k$-th elementary symmetric function of the eigenvalues of the Hessian.  It is known that the $k$-Hessian equation $\sigma_k(D^2u)=f$ with Dirichlet boundary condition $u=0$ is variational; indeed, this problem can be studied by means of the $k$-Hessian energy $-\int u\sigma_k(D^2u)$.  We construct a natural boundary functional which, when added to the $k$-Hessian energy, yields as its critical points solutions of $k$-Hessian equations with general non-vanishing boundary data. As a consequence, we prove a sharp Sobolev trace inequality for $k$-admissible functions $u$ which estimates the $k$-Hessian energy in terms of the boundary values of $u$.\end{abstract}
\maketitle

\section{Introduction}
\label{sec:intro}

Let $X\subset\bR^n$ be a bounded smooth domain with boundary $M=\partial X$.  The usual sharp Sobolev trace inequality states that
\begin{equation}
 \label{eqn:model_trace}
 -\int_X u\Delta u\, dx + \oint_M fu_n d\mu\geq \oint_M f(u_f)_n d\mu
\end{equation}
for all $f\in C^\infty(M)$ and all $u\in C^\infty(\overline{X})$ such that $u\rv_M=f$, where $u_n$ denotes the derivative of $u$ with respect to the outward-pointing normal along $M$, $u_f$ is the harmonic function in $X$ such that $u_f\rv_M=f$, and $dx$, $d\mu$ are the volume forms on $X$ and $M$, respectively. A standard density argument implies that the trace $u\mapsto u\rv_M=:\trace u$ extends to a bounded linear operator $\trace\colon W^{1,2}(\overline{X})\to W^{1/2,2}(M)$, while the extension $f\mapsto u_f=:E(f)$ extends to a bounded linear operator $E\colon W^{1/2,2}(M)\to W^{1,2}(\overline{X})$ such that $\trace\circ E$ is the identity.

The sharp Sobolev trace inequality~\eqref{eqn:model_trace} is a useful tool in many analytic and geometric problems.  For example, the Dirichlet-to-Neumann map $f\mapsto(u_f)_n$ is a pseudodifferential operator with principle symbol $(-\Delta)^{1/2}$; indeed, it is the operator $(-\Delta)^{1/2}$ when $\Omega=\bR_+^n$ is the upper half-plane.  Thus~\eqref{eqn:model_trace} relates the energy of the local operator $-\Delta$ to the energy of the nonlocal Dirichlet-to-Neumann operator, providing a useful tool for establishing estimates for PDEs stated in terms of the latter operator.  This strategy provides a key motivation for the approach of Caffarelli and Silvestre~\cite{CaffarelliSilvestre2007} for studying fractional powers of the Laplacian.  As another example, Escobar~\cites{Escobar1988,Escobar1990} proved an analogue of~\eqref{eqn:model_trace} on compact manifolds with boundary for which both sides of the inequality are conformally invariant.  In particular, this recovers~\eqref{eqn:model_trace} when $X=\bR_+^n$.  Using conformal invariance, he also proved a sharp Sobolev trace inequality which yields the continuous embedding $W^{1,2}(\overline{\bR_+^n})\subset L^{\frac{2(n-1)}{n-2}}(\bR^{n-1})$ when $n\geq3$.  This work has important implications for the Yamabe Problem on manifolds with boundary~\cite{Escobar1992}.  By considering weights or higher-order operators, analogues of~\eqref{eqn:model_trace} have been established with implications for the energies of fractional powers of the Laplacian of all non-integral orders~\cites{CaffarelliSilvestre2007,Yang2013} as well as for the energies of conformally covariant fractional powers of the Laplacian~\cites{Case2015fi,CaseChang2013,ChangGonzalez2011,ChangYang2015} and the fractional Yamabe problem~\cite{GonzalezQing2013}.

The purpose of this article is to establish an analogue of~\eqref{eqn:model_trace} in terms of the $k$-Hessian energy $\sigma_k(D^2u)$.  Here $D^2u$ denotes the Hessian of $u$ and the $k$-th elementary symmetric function $\sigma_k(A)$ of a symmetric matrix $A$ is defined by
\[ \sigma_k(A) := \sum_{i_1<\dotsb<i_k} \lambda_{i_1}\dotsm\lambda_{i_k} \]
for $\lambda_1,\dotsc,\lambda_n$ the eigenvalues of $A$.  The Dirichlet problem
\begin{equation}
 \label{eqn:dirichlet_problem}
 \begin{cases}
  \sigma_k(D^2u) = F(x,u), & \text{in $X$}, \\
  u = f(x), & \text{on $M$}
 \end{cases}
\end{equation}
has been well-studied for functions $u$ in the elliptic $k$-cone
\begin{equation}
 \label{eqn:kcone}
 \Gamma_k^+ := \left\{ u\in C^\infty(\oX) \suchthat \sigma_j(D^2u)>0, 1\leq j\leq k \right\} ;
\end{equation}
e.g.\ \cites{Caff4,IvochkinaTrudingerWang2004,Urbas1990,Wang2,Wang}.  Note that the existence of a solution to~\eqref{eqn:dirichlet_problem} requires that $M$ be $(k-1)$-convex~\cite{Caff4}; i.e.\ the second fundamental form $L$ of $M$ must satisfy $\sigma_j(L)>0$ for $1\leq j\leq k-1$.  Indeed, provided $M$ is $(k-1)$-convex, X.-J.\ Wang proved~\cite{Wang2} the fully nonlinear Sobolev inequality
\begin{equation}
 \label{eqn:wang_sobolev}
 \int_X -u\sigma_k(D^2u) dx\geq C(X)\left(\int_X \lv u\rv^{\frac{n(k+1)}{n-2k}}dx\right)^{\frac{n-2k}{n}}
\end{equation}
for all $u\in\Gamma_k^+$ such that $u\rv_M=0$.  In a sense, the Sobolev inequality~\eqref{eqn:wang_sobolev} is dual to the desired fully nonlinear analogue of~\eqref{eqn:model_trace}: in~\eqref{eqn:wang_sobolev} the extremal functions are ``flat'' on the boundary, in the sense $u\rv_M=0$, while in~\eqref{eqn:model_trace} the extremal functions are ``flat'' in the interior, in that $\Delta u=0$.

To establish a fully nonlinear analogue of~\eqref{eqn:model_trace} requires us to both know that the purported minimizers of the inequality exist and to identify what boundary terms to add to the interior term $-\int u\sigma_k(D^2u) dx$.  The first problem is settled: existence and uniqueness of a solution $u\in\overline{\Gamma_k^+}$ of the degenerate Dirichlet problem~\eqref{eqn:dirichlet_problem} with $F=0$ is known~\cites{IvochkinaTrudingerWang2004,WangXu}; here $\overline{\Gamma_k^+}$ is the closure of the elliptic $k$-cone~\eqref{eqn:kcone} with respect to the $C^{1,1}$-norm in $\oX$.  The second problem is addressed in this article.  This is accomplished via the following proposition.

\begin{proposition}
 \label{prop:main_prop}
 Let $X\subset\bR^n$ be a bounded smooth domain with boundary $M=\partial X$ and let $k\in\bN$.  Then there is a multilinear differential operator
 \begin{equation}
  \label{eqn:boundary_requirement}
  B_k\colon\left(C^1(\overline{X}) \cap C^2(M)\right)^k \to C^0(M)
 \end{equation}
 such that the multilinear form $\mQ_k\colon\left(C^2(X)\cap C^2(M)\cap C^1(\overline{X})\right)^{k+1}\to\bR$ defined by
 \begin{equation}
  \label{eqn:k-energy}
  \mQ_k(u,w^1,\dotsc,w^k) := -\int_X u\,\sigma_k(D^2w^1,\dotsc,D^2w^k)dx + \oint_M u\,B_k(w^1,\dotsc,w^k)d\mu
 \end{equation}
 is symmetric, where $\sigma_k(D^2w^1,\dotsc,D^2w^k)$ is the polarization of the $k$-linear map $w\mapsto\sigma_k(D^2w)$.
\end{proposition}

\begin{remark}
 The notation~\eqref{eqn:boundary_requirement} specifies that the operators $B_k$ depend on at most second-order tangential derivatives and at most first-order transverse derivatives of their inputs along the boundary $M$.
\end{remark}

An explicit formula for such operators $B_k$ can be deduced from Section~\ref{sec:polarized} and Section~\ref{sec:adjusted}.  From~\eqref{eqn:model_trace} we see that $B_1(u)=u_n$ satisfies the conclusions of Proposition~\ref{prop:main_prop}.  The following result gives a boundary operator which satisfies the conclusions of Proposition~\ref{prop:main_prop} when $k=1$.

\begin{proposition}
 \label{prop:main_prop2}
 Let $X\subset\bR^n$ be a bounded smooth domain with boundary $M=\partial X$.  Define $B_2\colon\left(C^1(\oX)\cap C^2(M)\right)^2\to C^0(M)$ by
 \begin{equation}
  \label{eqn:k2_boundary_operator}
  B_2(v,w) = \frac{1}{2}\left(v_n\oDelta w + w_n\oDelta v + L(\onabla v,\onabla w) + Hv_nw_n \right) .
 \end{equation}
 Then the multilinear form $\mQ_2\colon\left(C^2(\oX)\right)^3\to\bR$ given by
 \[ \mQ_2(u,v,w) = -\int_X u\sigma_2(D^2v,D^2w)dx + \oint_M uB_2(v,w)d\mu \]
 is symmetric.
\end{proposition}

Here $\oDelta$ and $\onabla$ denote the tangential Laplacian and tangential gradient, respectively; i.e.\ the Laplacian and the gradient defined with respect to the induced metric on the boundary $M$.

Denote by $\mE_k(u):=\mQ_k(u,\dotsc,u)$ the energy associated to $\mQ_k$ as in Proposition~\ref{prop:main_prop}.  The fact that~\eqref{eqn:k-energy} defines a symmetric $(k+1)$-linear form implies that if $v\in C^\infty(\overline{X})$ is such that $v\rv_M=0$, then
\[ \left.\frac{d^j}{dt^j}\right|_{t=0}\mE_k(u+tv) = -\frac{(k+1)!}{(k+1-j)!}\int_X v\,\sigma_k\Bigl(\overbrace{D^2v,\dotsc,D^2v}^{j-1},\overbrace{D^2u,\dotsc,D^2u}^{k+1-j}\Bigr) dx \]
for all $1\leq j\leq k+1$.  That is, within a class
\[ \mC_f := \left\{ u\in C^\infty(\overline{X}) \suchthat u\rv_M = f \right\} \]
of functions with fixed trace $f\in C^\infty(M)$, the derivatives of the energies $\mE_k$ depend only on the interior integrals.  In particular, it is straightforward to identify the critical points of $\mE_k$ and deduce the convexity of $\mE_k$ within the positive cone $\Gamma_k^+$.  This leads to the following family of fully nonlinear Sobolev trace inequalities.

\begin{theorem}
 \label{thm:main_thm}
 Fix $k\in\bN$ and let $X\subset\bR^n$ be a bounded $(k-1)$-convex domain with boundary $M=\partial X$.  Let $B_k$ be as in Proposition~\ref{prop:main_prop}.  Given $f\in C^\infty(M)$, let
 \[ \mC_{f,k} := \left\{ u\in \mC_f \suchthat D^2u\in\Gamma_k^+ \right\} . \]
 Then it holds that
 \begin{equation}
  \label{eqn:k-sobolev-trace}
  \mE_k(u) \geq \mE_k(u_f)
 \end{equation}
 for all $u\in\overline{\mC_{f,k}}$, where $u_f$ is the unique solution to the Dirichlet problem
 \begin{equation}
  \label{eqn:degenerate_dirichlet_problem}
  \begin{cases}
   \sigma_k(D^2u) = 0, & \text{in $X$}, \\
   u = f, & \text{on $M$} ,
  \end{cases}
 \end{equation}
 and $\overline{\mC_{f,k}}$ is the closure of $\mC_{f,k}$ with respect to the $C^{1,1}$-norm in $\oX$.
\end{theorem}

Note that $\mE_k(u_f)=\oint f\,B_k(u_f,\dotsc,u_f)d\mu$, so that Proposition~\ref{prop:main_prop} implies that the right-hand side of~\eqref{eqn:k-sobolev-trace} depends only on $f$, the tangential gradient $\onabla f$, the tangential Hessian $\bar D^2 f$, and the normal derivative $(u_f)_n$ of the extension $u_f$.  This is consistent with the expected regularity $u_f\in C^{1,1}(\overline{X})$.  One may regard~\eqref{eqn:k-sobolev-trace} as a norm inequality for part of the trace embedding $W^{\frac{2k}{k+1},k+1}(X)\subset W^{\frac{2k-1}{k+1},k+1}(M)$.

We conclude this introduction with a few additional comments on the boundary operators $B_k$ of Proposition~\ref{prop:main_prop}.  Given $f\in C^\infty(M)$ and $k\in\bN$, define
\[ \mB_k(f) := B_k(u_f,\dotsc,u_f) \]
for $u_f$ the solution to~\eqref{eqn:degenerate_dirichlet_problem}.  The specification~\eqref{eqn:boundary_requirement} of the domain of the boundary operators $B_k$ implies that $\mB_k$ is a well-defined function; it should be regarded as a fully nonlinear analogue of the Dirichlet-to-Neumann map.  Theorem~\ref{thm:main_thm} yields a relationship between the energy of $\mB_k$ and the energy associated to the $\sigma_k$-curvature.  Motivated by the similar relationship between the energies associated to fractional order operators and the Laplacian induced by~\eqref{eqn:model_trace}, we propose the study of the operators $\mB_k$ as an interesting family of fully nonlinear pseudodifferential operators.  In particular, it seems interesting to ask if there exists a constant $C(M)>0$ such that
\[ A(M)\oint_M f\,\mB_k(f)d\mu + B(M)\oint_M \lv f\rv^{k+1} d\mu\geq \left( \oint_M \lv f\rv^{\frac{(k+1)(n-1)}{n-2k}}d\mu\right)^{\frac{n-2k}{n-1}} . \]
If true, this would provide a fully nonlinear analogue of the sharp Sobolev inequality of X.-J.\ Wang~\cite{Wang}.  Note that this is already known in the case $k=1$; cf.\ \cite{LiZhu1997}.

The conditions of Proposition~\ref{prop:main_prop} do not uniquely determine the boundary operators $B_k$ of Proposition~\ref{prop:main_prop}; indeed, the operators are not unique even if we require additionally that the operators $B_k$ commute with diffeomorphisms, as do the operators constructed in the proof of Proposition~\ref{prop:main_prop}.  A trivial source of nonuniqueness comes from the freedom to add symmetric zeroth-order terms to $B_k$.  For example, if $B_k$ satisfies the conclusions of Proposition~\ref{prop:main_prop}, so too does the operator
\[ (w^1,\dotsc,w^k) \mapsto B_k(w^1,\dotsc,w^k) + cHw^1\dotsm w^k \]
for any $c\in\bR$, where $H$ is the mean curvature of the boundary $M$.  More generally, one may add to the boundary operators $B_k$ any symmetric multilinear operator which is also symmetric upon pairing with integration.  For example, consider the operator $D\colon\left(C^1(\overline{X})\right)^2\to C^\infty(M)$ defined by
\[ D(v,w) = \odelta\left( L(\onabla(vw))\right) - L(\onabla v,\onabla w) . \]
It is readily verified that $(u,v,w)\mapsto\oint u\,D(v,w)d\mu$ is a symmetric trilinear form, and thus $D$ can be added to the operator~\eqref{eqn:k2_boundary_operator} to yield another operator $\tilde B_2$ which satisfies the conclusions of Proposition~\ref{prop:main_prop}.

This article is organized as follows.  In Section~\ref{sec:bg} we collect some useful facts involving the $k$-Hessian and the elliptic cones.  In Section~\ref{sec:polarized} and Section~\ref{sec:adjusted} we prove Proposition~\ref{prop:main_prop} by explicitly constructing a suitable boundary operator.  In Section~\ref{sec:variation} we prove Theorem~\ref{thm:main_thm}.  In Section~\ref{sec:example} we discuss in more detail the case $k=2$.

\section{Preliminaries}
\label{sec:bg}

\subsection{The $\Gamma_k^+$-cone}
In this subsection, we describe some properties of the elementary symmetric functions and their associated convex cones.

\begin{definition} The $k$-th elementary symmetric function for
 $\lambda=(\lambda_1,\dotsc,\lambda_n)\in \mathbb{R}^n$ is
$$\sigma_{k}(\lambda):=\sum_{i_1<\dotsb<i_k}\lambda_{i_1}\dotsm \lambda_{i_k}.$$
\end{definition}

The elementary symmetric functions are special cases of hyperbolic polynomials~\cite{Garding}.  As such, they enjoy many nice properties in their associated positive cones.

\begin{definition}\label{cone}
The positive $k$-cone is the connected component of $\left\{\lambda \suchthat\sigma_k(\lambda)>0\right\}$ which contains $(1,\dotsc,1)$.  Equivalently,
\[ \Gamma_k^+=\left\{\lambda\in \mathbb{R}^n\suchthat \sigma_{1}(\lambda)>0,\dotsc,\sigma_{k}(\lambda)>0\right\}. \]
\end{definition}

For example, the positive $n$-cone is
\[ \Gamma_n^+ = \left\{ \lambda\in\bR^n \suchthat \lambda_1,\dotsc,\lambda_n>0 \right\} \]
and the positive $1$-cone is the half-space
\[ \Gamma_1^+ = \left\{ \lambda\in\bR^n \suchthat \lambda_1+\dotsb+\lambda_n>0 \right\} . \]
Note that $\Gamma_k^+$ is an open convex cone and that
$$\Gamma_{n}^+ \subset \Gamma_{n-1}^+\dotsb \subset\Gamma_{1}^+.$$

Applying G{\aa}rding's theory of hyperbolic polynomials \cite{Garding}, one concludes that
$\sigma_{k}^{\frac{1}{k}}$ is a concave function in $\Gamma_{k}^{+}$.

\begin{definition}A symmetric matrix $A$ is in the $\tilde{\Gamma}_{k}^+$ cone if its 
eigenvalues $$\lambda(A)=(\lambda_1(A),\dotsc,\lambda_n(A))\in  \Gamma_{k}^+.$$
\end{definition}
Suppose $f$ is a function on $\Gamma_{k}^+$. Denote by $F=f(\lambda(A))$ the function on $\tilde{\Gamma}_{k}^+$ induced by $f$.  It is known~\cite{Caff4} that if $f$ is concave in $\Gamma_{k}^{+}$, then the induced function $F$ is concave in $\tilde{\Gamma}_{k}^{+}$.  For this reason, we shall denote $\tilde{\Gamma}_k^+$ by $\Gamma_k^+$ and $\sigma_k(\lambda(A))$ by $\sigma_k(A)$ when there is no possibility of confusion.

Notice that $\sigma_n(A)=\det (A)$. An equivalent definition of $\det (A)$ is
\[ \det A:=\frac{1}{n!}\delta^{i_1\dotsc i_n}_{j_1\dotsc j_n}A_{i_1j_1} \dotsm A_{i_nj_n}, \]
where $\delta^{i_1\dotsc i_n}_{j_1\dotsc j_n} $ is the generalized Kronecker delta; it is zero if $\{i_1,\dotsc,i_n\}\neq\{j_1,\dotsc,j_n\}$ and equals $1$ (resp.\ equals $-1$) if $(i_1,\dotsc,i_n)$ and $(j_1,\dotsc,j_n)$ differ by an even (resp.\ odd) permutation.  Similarly, an equivalent definition of $\sigma_k(A)$ is
$$\sigma_{k}(A):=\frac{1}{k!}\delta^{i_1\dotsc i_k}_{j_1\dotsc j_k}A_{i_1j_1} \dotsm A_{i_kj_k}.$$
The Newton transformation tensor is defined as
\[ T_{k}(A)_{ij}:=\frac{1}{k!}
\delta^{ii_1\dotsc i_k}_{jj_1\dotsc j_k}{(A)}_{i_1j_1} \dotsm {(A)}_{i_kj_k}.\]

\begin{definition}\label{2.3}
The polarization of $\sigma_k$ is
\[ \sigma_{k}(A_1,\dotsc,A_k):=\frac{1}{k!}\delta^{i_1\dotsc i_k}_{j_1\dotsc j_k}{(A_1)}_{i_1j_1} \dotsm {(A_k)}_{i_kj_k}.\]
\end{definition}

It is called the polarization of $\sigma_k$ because $\sigma_k(A_1,\dotsc,A_k)$ is the symmetric multilinear form such that $\sigma_k(A)=\sigma_k(A,\dotsc,A)$.

\begin{definition}\label{2.4}
The polarized Newton transformation tensor is
\[ T_{k}(A_1,\dotsc,A_k)_{ij}:=\frac{1}{k!}\delta^{ii_1\dotsc i_k}_{jj_1\dotsc j_k}{(A_1)}_{i_1j_1} \dotsm {(A_k)}_{i_kj_k}. \]
\end{definition}

When some components in the polarizations are the same, we adopt the notational conventions
\begin{align*}
 \sigma_{k}(\overbrace{B,\dotsc,B}^{l},C,\dotsc,C) & :=\sigma_{k}(\overbrace{B,\dotsc,B}^{l},\overbrace{C,\dotsc,C}^{k-l}), \\
 T_{k}(\overbrace{B,\dotsc,B}^{l},C,\dotsc,C)_{ij} & :=T_{k}(\overbrace{B,\dotsc,B}^{l},\overbrace{C,\dotsc,C}^{k-l})_{ij}.
\end{align*}

Some useful relations between the Newton transformation tensor $T_{k}$ and $\sigma_k$ are as follows.
For any symmetric matrix $A$, if we denote the trace by $\Tr$, then
\begin{align*}
 \sigma_{k}(A)&=\frac{1}{n-k}\Tr(T_{k}(A)_{ij}), \\
 \sigma_{k+1}(A)&=\frac{1}{k+1}\Tr(T_{k}(A)_{im}A_{mj} ).
\end{align*}

Many useful algebraic inequalities for elements of $\Gamma_{k}^+$ can be deduced from G{\aa}rding's theory of hyperbolic polynomials~\cite{Garding}.  For us, the important such inequality is the fact that if $A_1,\dotsc,A_k\in\bar{\Gamma}_{k+1}^{+}$, then $T_k(A_1,\dots,A_k)_{ij}$ is a nonnegative matrix.

\section{Construction of the polarized functional}
\label{sec:polarized}

We begin our construction of the boundary integrals of Proposition~\ref{prop:main_prop}.  Define
\begin{equation}
 \label{eqn:S0}
 \begin{split}
  S_0(u,w^1,\dotsc,w^k)&:= -2\sum_p \int_X u_i w_j^p T_{k-1}(D^2 w^{\w p})_{ij}dx\\
  &\quad-\sum_{p\neq q} \int_X w_i^p w_j^q T_{k-1}(D^2 u, D^2 w^{\w p,q})_{ij}dx.\\
 \end{split}
\end{equation}
where $D^2w^{\w p}$ denotes the list $(D^2w^1,\dotsc,D^2w^{p-1},D^2w^{p+1},\dotsc,D^2w^k)$ obtained from $(D^2w^1,\dotsc,D^2w^k)$ by removing the entry $D^2w^p$, and likewise $D^2 w^{\w p,q}$ denotes the list obtained from $(D^2w^1,\dotsc,D^2w^k)$ by removing the entries $D^2w^p$ and $D^2w^q$.  Similar notation will be used to remove more elements from the list.  Using integration by parts to rewrite~\eqref{eqn:S0} as a sum of an interior and a boundary integral, both of which have integrands which factor through $u$, yields the following first step towards proving Proposition~\ref{prop:main_prop}.

\begin{proposition}
 \label{prop:pre_main_prop}
 There exists a symmetric $\bR$-multilinear function $A_k\colon C^\infty(\overline{X})\to C^\infty(M)$ such that
\begin{equation}
 \label{eqn:L}
 L(u,w^1,\dotsc, w^k):= \int_X u \sigma_k(D^2 w^1,\dotsc, D^2 w^k) dx + \oint_M u A_k(w^1,\dotsc,w^k)d\mu
\end{equation}
is symmetric in  $u,w^1,\dotsc, w^k$.
\end{proposition}

\begin{remark}
 The operators $A_k$ constructed by our proof depend on at most $4$ derivatives of their inputs.
\end{remark}

\begin{proof}
Note that $S_0$ is symmetric.  Our objective is to rewrite~\eqref{eqn:S0} in the desired form~\eqref{eqn:L}.  To that end, writing~\eqref{eqn:S0} as a sum over pairs $p\not=q$ and then integrating by parts in $X$ yields
\begin{align*}
S_0&=\sum_{p\neq q}\Bigl[-\frac{2}{k-1}
\int_X u_i w_j^p T_{k-1}(D^2 w^{\w p})_{ij} dx  -\int_X w_i^p w_j^q T_{k-1}(D^2 u, D^2 w^{\w p,q})_{ij} dx \Bigr]\\
&=\sum_{p\neq q}\Bigl[\frac{2}{k-1}
\int_X u w_{ij}^p T_{k-1}(D^2 w^{\w p})_{ij}dx+\int_X w^p u_{ij} T_{k-1}( D^2 w^{\w p})_{ij}dx\\
&\quad-\frac{2}{k-1}\oint_M u w_{j}^p T_{k-1}(D^2 w^{\w p})_{jn} d\mu -\oint_M w^p w_{j}^q T_{k-1}(D^2 u, D^2 w^{\w p,q })_{jn} d\mu\Bigr].
\end{align*}
Integrating by parts in $X$ once more yields
\begin{align*}
S_0&=\sum_{p\neq q}\Bigl[\frac{k+1}{k-1}
\int_X u w_{ij}^p T_{k-1}(D^2 w^{\w p})_{ij}dx-\frac{k+1}{k-1}\oint_M u w_{j}^p T_{k-1}(D^2 w^{\w p})_{jn}d\mu \\
&\quad-\oint_M w^p w_{j}^q T_{k-1}(D^2 u, D^2 w^{\w p,q })_{jn}d\mu
+\oint_M w^p u_j T_{k-1}(D^2 w^{\w p})_{jn}d\mu\Bigr].
\end{align*}
Denote the boundary integral by $T$:
\begin{equation}
 \label{6.1}
 \begin{split}
  T &= \sum_{p\neq q}\Bigl[\oint_M w^p u_j T_{k-1}( D^2 w^{\w p})_{jn} d\mu- \oint_M w^p w_j^q T_{k-1}(D^2 u, D^2 w^{\w p,q})_{jn}d\mu \Bigr] \\
   &\quad-(k+1)\sum_p\oint_M uw_j^pT_{k-1}(D^2w^{\w p})_{jn} d\mu.
 \end{split}
\end{equation}
Thus 
\[ S_0 = k^2(k+1) \int_X u \sigma_{k}(D^2 w^{1},\dotsc, D^2 w^k)dx + T . \]

We aim to write $T$ as the sum of a symmetric term and a boundary integral of the form $\oint uB(w^1,\dotsc, w^k)d\mu$.  To that end, consider the symmetrization of the second term of~\eqref{6.1}:
\begin{equation}\label{6.2}
\begin{split}
S_1&:=\sum_{p\neq q}\Bigl[ - \oint_M w^p w_j^q T_{k-1}(D^2 u, D^2 w^{\w p,q})_{jn}d\mu\\
&\quad-\frac{1}{k-1}\oint_M w^p u_j T_{k-1}( D^2 w^{\w p})_{jn}d\mu-\frac{1}{k-1}\oint_M  u w^p_j T_{k-1}( D^2 w^{\w p})_{jn}d\mu\Bigr].\\
\end{split}
\end{equation}
Note that $S_1$ is symmetric with respect to $u, w^1, \dotsc, w^k$.  Combining~\eqref{6.1} and~\eqref{6.2} yields
\[ T=S_1-k\sum_{p}\oint_M u w_j^p T_{k-1}(D^2 w^{\w p})_{jn} d\mu+ k\sum_{p}\oint_M w^p u_j T_{k-1}( D^2 w^{\w p})_{jn}d\mu . \]
We define
\begin{align*}
U_1 & := -k\sum_{p}\oint_M u w_j^p T_{k-1}(D^2 w^{\w p})_{jn} d\mu, \\
Q & := k\sum_{p} \oint w^p u_j T_{k-1}( D^2 w^{\w p})_{jn} d\mu,
\end{align*}
so that
\[ T= U_1+ S_1+ Q. \]
$U_1$ is of the correct form  $\oint u B(w^1,\cdots w^p) d\mu$.  We continue with the term $Q$.  Observe that
\[ Q=k\sum_{p}\Bigl[\oint_M w^p u_{\alpha} T_{k-1}( D^2 w^{\w p})_{\alpha n}d\mu + \oint_M w^p u_{n} T_{k-1}( D^2 w^{\w p})_{n n} d\mu\Bigr] . \]
where Greek indices $\alpha,\beta\in\{1,\dotsc,n-1\}$ denote tangential directions and $n$ denotes the outward-pointing normal along $M$.  By the definition of Newton tensor, $T_{k-1}( D^2 w^{\w p})_{n n}=\sigma_{k-1}(D^2 w\rv_{TM}^{\w p} )$, where $D^2w\rv_{TM}^{\w p}$ denotes the list of the restrictions $D^2w^1\rv_{TM},\dotsc,D^2w^n\rv_{TM}$ with the $p$-th element removed. Thus 
\[ Q=k\sum_{p}\Bigl[ \oint_M w^p u_{\alpha} T_{k-1}( D^2 w^{\w p})_{\alpha n}d\mu + \oint_M w^p u_{n} \sigma_{k-1}(D^2 w\rv_{TM}^{\w p} ) d\mu\Bigr] . \]
Define 
\begin{align*}
 U_2 & := k\sum_{p} \oint_M w^p u_\alpha T_{k-1}(D^2w^{\w p})_{\alpha n}d\mu, \\
 Q_1 & := k\sum_{p} \oint_M w^pu_{n} \sigma_{k-1}(D^2w\rv_{TM}^{\w p}) d\mu,
\end{align*}
so that
\[ Q = U_2 + Q_1 . \]
Integrating by parts along $M$ shows that 
\begin{equation*}U_2= -k\sum_{p} \oint u ( w^p T_{k-1}( D^2 w^{\w p})_{\alpha n})_\alpha d\mu.\\
\end{equation*}
Thus $U_2$ is of the correct form $\oint uB(w^1,\dotsc,w^p)d\mu$. Therefore we need only consider $Q_1$.

Consider the symmetrization of $Q_1$:
\begin{align*}
S_2&:=\sum_{p\neq q}\Bigl[ \frac{k}{k-1}\oint_M w^p u_{n} \sigma_{k-1}(D^2 w\rv_{T_xM}^{\w p} )d\mu + \frac{k}{k-1}\oint_M u w^p_{n} \sigma_{k-1}(D^2 w\rv_{TM}^{\w p} )d\mu\\
&\quad+k \oint_M w^p w^q_{n} \sigma_{k-1}(D^2 u\rv_{TM} , D^2 w\rv_{TM}^{\w p,q} ) d\mu\Bigr].\\
\end{align*}
Note that $S_2$ is symmetric with respect to $u, w^1, \dotsc, w^k$.  Moreover,
\begin{equation}\label{eqn:Q1_sym}
\begin{split}
Q_1 & =S_2 - \frac{k}{k-1}\sum_{p\neq q} \oint_M u w^p_{n} \sigma_{k-1}(D^2 w\rv_{TM}^{\w p} ) d\mu\\
& \quad - k\sum_{p\neq q} \oint_M w^p w^q_{n} \sigma_{k-1}(D^2 u\rv_{TM} , D^2 w\rv_{TM}^{\w p,q} )d\mu .
\end{split}
\end{equation}
Denote by $\bar{D}^2$ the Hessian with respect to the induced metric of $M$ and by $L_{\alpha\beta}$ the second fundamental form of $M$. Given $v\in C^\infty(\oX)$, it holds that
\begin{equation}
 \label{eqn:tangential_hessian}
 D^2 v\rv_{TM} = \bar{D}^2 v + v_nL
\end{equation}
along $M$.  Define
\begin{align*}
 U_3 & :=-\frac{k}{k-1}\sum_{p\neq q} \oint_M u w^p_{n} \sigma_{k-1}(D^2 w\rv_{TM}^{\w p} ) d\mu, \\
 U_4 & :=-k\sum_{p\neq q} \oint_M w^p w^q_{n} \sigma_{k-1}(\bar{D}^2 u , D^2 w\rv_{TM}^{\w p,q} )d\mu .
\end{align*}
Integrating by parts along $M$ yields
\[ U_4=-\frac{k}{k-1}\sum_{p\neq q} \oint_M u \big(w^p w^q_{n} T_{k-2}( D^2 w\rv_{TM}^{\w p,q} )_{\alpha\beta}\big)_{\bar\alpha\bar\beta} d\mu , \]
where the bars on $\alpha$ and $\beta$ denote covariant derivatives with respect to the induced metric on $M$.  In particular, both $U_3$ and $U_4$ are of the form $\oint uB(w^1,\dotsc,w^k)d\mu$.  Define
\[ Q_2:=-k\sum_{p\neq q} \oint_M w^p w^q_{n}u_n \sigma_{k-1}(L, D^2 w\rv_{TM}^{\w p,q} ) d\mu. \]
It follows from~\eqref{eqn:Q1_sym},~\eqref{eqn:tangential_hessian} and the definitions of $U_3, U_4, Q_2$  that
\[ Q_1=S_2+U_3+U_4+Q_2 . \]

Now we want to write $Q_2$ in the desired form.  To that end, consider the symmetrization of $Q_2$:
\begin{equation}\label{6.4}
\begin{split}
S_3&:=-k \sum_{p\neq q\neq r} \Bigl[\frac{1}{k-2}\oint_M w^p w^q_{n}u_n \sigma_{k-1}(L, D^2 w\rv_{TM}^{\w p,q} )d\mu \\
&\quad+\frac{1}{2!(k-2)}\oint_M u w^p_n w^q_{n} \sigma_{k-1}(L, D^2 w\rv_{TM}^{\w p,q} ) d\mu\\
&\quad+\frac{1}{2!}\oint_M w^p w^q_{n}w^r_n \sigma_{k-1}( L,D^2 u\rv_{TM}, D^2 w\rv_{TM}^{\w p,q, r} )d\mu \Bigr],\\
\end{split}
\end{equation}
Note that $S_3$ is symmetric with respect to $u,w^1,\dotsc,w^k$.  Define
\begin{align*}
 U_5 & :=\frac{k}{2!(k-2)}\sum_{p\neq q\neq r}\oint_M u w^p_n w^q_{n} \sigma_{k-1}(L, D^2 w\rv_{TM}^{\w p,q} ) d\mu,\\
 U_6 & :=\frac{k}{2!}\sum_{p\neq q\neq r}\oint_M w^p w^q_{n}w^r_n \sigma_{k-1}( L,\bar{D}^2 u\rv_{TM}, D^2 w\rv_{TM}^{\w p,q,r} ) d\mu.
\end{align*}
As above, integration by parts along $M$ implies that both $U_5$ and $U_6$ are of the form $\oint uB(w^1,\dotsc,w^k)d\mu$.  Define
\[ Q_3 := \frac{k}{2!}\sum_{p\neq q\neq r}\oint_M w^p w^q_{n}w^r_nu_n \sigma_{k-1}( L,L, D^2 w\rv_{TM}^{\w p,q,r} )d\mu . \]
From~\eqref{6.4} and the definitions of $Q_2$, $U_5$, $U_6$ and $Q_3$ we deduce that
\[ Q_2= S_3+U_5+U_6+Q_3 . \]

Proceeding in this way, for all $2\leq i\leq k$ we make the following definitions.  First, define
\begin{align*}
 S_i & := (-1)^ik \sum_{p_1\neq \dotsb \neq p_i}  \Bigl[\frac{1}{(i-2)!(k+1-i)}\oint_M w^{p_1}  w^{p_2}_{n}\dotsm w^{p_{i-1}}_n u_n \\
  &\qquad\qquad\qquad\qquad \times \sigma_{k-1}(\overbrace {L,\dotsc, L}^{i-2}, D^2 w\rv_{TM}^{\w p_1,\dotsc, p_{i-1}} ) d\mu\\
  &\quad+\frac{1}{(i-1)!(k+1-i)}\oint_M u w^{p_1}_n \dotsm w^{p_{i-1}}_n \sigma_{k-1}(\overbrace {L,\dotsc, L}^{i-2}, D^2 w\rv_{TM}^{\w p_1,\dotsc, p_{i-1}} )d\mu ,\\
  &\quad+\frac{1}{(i-1)!}\oint_M w^{p_1}  w^{p_2}_{n}\dotsm w^{p_{i}}_n  \sigma_{k-1}( \overbrace {L,\dotsc, L}^{i-2},D^2 u\rv_{TM}, D^2 w\rv_{TM}^{\w p_1,\dotsc, p_i} ) d\mu\Bigr] .
\end{align*}
Note that $S_i$ is symmetric with respect to $u,w^1,\dotsc,w^k$.  Next, define
\begin{align*}
 U_{2i-1}& := \frac{(-1)^{i+1}k}{(i-1)!(k+1-i)}\sum_{p_1\neq \dotsb \neq p_i}\oint_M u w^{p_1}_n \dotsm w^{p_{i-1}}_n  \\
&\quad \quad  \quad \quad \quad \quad \quad  \quad \quad \quad \quad \quad \quad  \quad \quad \times \sigma_{k-1}(\overbrace {L,\dotsc, L}^{i-2}, D^2 w\rv_{TM}^{\w p_1,\dotsc, p_{i-1}} ) d\mu,\\
 U_{2i}& := \frac{(-1)^{i+1}k}{(i-1)!}\sum_{p_1\neq \dotsb \neq p_i} \oint_M w^{p_1}  w^{p_2}_{n}\dotsm w^{p_i}_n \sigma_{k-1}(\overbrace {L,\dotsc, L}^{i-2},\bar{D}^2u, D^2 w\rv_{TM}^{\w p_1,\dotsc, p_i})d\mu .
\end{align*}
Integration by parts along $M$ implies that both $U_{2i-1}$ and $U_{2i}$ are of the form $\oint uB(w^1,\dotsc,w^k)d\mu$.  Then
\[ Q_i:= \frac{(-1)^{i+1} k}{(i-1)!}\sum_{p_1\neq \dotsb \neq p_i}\oint_M w^{p_1} w^{p_2}_{n}\dotsm w^{p_i}_n u_n \sigma_{k-1}(\overbrace {L,\dotsc, L}^{i-1}, D^2 w\rv_{TM}^{\w p_1,\dotsc,p_i} ) d\mu \]
is such that
\[ Q_{i-1} = S_i + U_{2i-1} + U_{2i} + Q_i . \]

It remains to write $Q_k$ as the sum of a symmetric integral and a boundary integral whose integrand factors through $u$.  To that end, define
\begin{align*}
 S_{k+1} & := \frac{(-1)^{k+1}k}{(k-1)!}\sum_{p_1\neq \dotsb \neq p_k} \Bigl[\oint_M w^{p_1}  w^{p_2}_{n}\dotsm w^{p_{k}}_n u_n\sigma_{k-1}(L)d\mu \\
 & \quad +\frac{1}{k}\oint_M u w^{p_1}_n \dotsm w^{p_{k}}_n \sigma_{k-1}(L )d\mu \Bigr] .
\end{align*}
Note that $S_{k+1}$ is symmetric with respect to $u,w^1,\dotsc,w^k$.  Also define
\[ U_{2k+1} := \frac{(-1)^{k}}{(k-1)!}\sum_{p_1\neq \dotsb \neq p_k}\oint_M u w^{p_1}_n \dotsm w^{p_{k}}_n \sigma_{k-1}(L)d\mu . \]
Note that $U_{2k+1}$ is of the form $\oint uB(w^1,\dotsc,w^k)d\mu$ and that
\[ Q_k = S_{k+1} + U_{2k+1} . \]

In summary, we have shown that
\begin{equation}
 \label{eqn:final_induction}
 S_0 - \sum_{i=1}^{k+1} S_i = k^2(k+1)\int_X u\,\sigma_{k}(D^2w^1,\dotsc,D^2w^{k}) dx+ \sum_{i=1}^{2k+1} U_i
\end{equation}
and observed that the left-hand side is symmetric in $u,w^1,\dotsc,w^k$ while the right-hand side is of the form $\oint uB(w^1,\dotsc,w^k)d\mu$.  Dividing~\eqref{eqn:final_induction} through by $k^2(k+1)$ yields~\eqref{eqn:L}.
\end{proof}

\section{Adjusted polarized functional}
\label{sec:adjusted}

The difference between Proposition~\ref{prop:main_prop} and Proposition~\ref{prop:pre_main_prop} is that in the latter result, we only ask that the boundary integrals making up the polarized functional are such that their integrands factor through $u$.  In particular, it is not clear that from the proof of Proposition~\ref{prop:pre_main_prop} that the functions $A_k$ depend only on at most second-order tangential derivatives and at most first-order transverse derivatives along $M$.  This arises in two ways.  First, the integral $U_1$ depends on the second-order derivative $w_{\alpha n}$.  Second, when written in the form $\oint uB(w^1,\dotsc,w^k) d\mu$, the integrals $U_{2i}$, $1\leq i\leq k$, depend also on third- and fourth-order derivatives of $w^p$.  By more carefully considering the integration by parts along $M$ invoked in the proof of Proposition~\ref{prop:pre_main_prop}, we show that the combination $\sum U_i$ only depends on at most second-order tangential derivatives and at most first-order transverse derivatives of $w^p$.  This proves Proposition~\ref{prop:main_prop}.  To that end, we first require a few facts.

\begin{lemma}
 \label{lem:boundary_derivatives}
 Let $X\subset\bR^n$ be a bounded smooth domain with boundary $M=\partial X$.  Let $w^1,\dotsc,w^k\in C^\infty(\oX)$.  Then
 \begin{align}
  \label{eqn:tangential_mixed} w_{\beta n} & = w_{n\bar\beta} - L_{\alpha\beta}w_\alpha, \\
  \label{eqn:Tktangential_mixed} T_k(D^2w^1,\dotsc,D^2w^k)_{\alpha n} & = -\frac{1}{k}\sum_{p=1}^k T_{k-1}(D^2w\rv_{TM}^{\w p})_{\alpha\beta} w_{\beta n}^p ,
 \end{align}
 where $\alpha,\beta\in\{1,\dotsc,n-1\}$ denote tangential directions, $n$ denotes the outward-pointing normal along the boundary, and $w_{n\bar\beta}$ denotes the tangential gradient of $w_n$.  Moreover,
 \begin{equation}
  \label{eqn:Tktangentialdivergence}
  T_k(\overbrace{L,\dotsc,L}^{i},D^2w\rv_{TM}^{\w p_1,\dotsc,p_i})_{\alpha\beta,\bar\beta} = \sum_{p\not=p_1,\dotsc,p_i} T_{k}(\overbrace{L,\dotsc,L}^{i+1},D^2w\rv_{TM}^{\w p,p_1,\dotsc,p_i})_{\alpha\beta}w_{\beta n}^p,
 \end{equation}
 where the left-hand side denotes the divergence with respect to the induced metric on $M$.
\end{lemma}

\begin{proof}
 \eqref{eqn:tangential_mixed} follows immediately from the definition of the second fundamental form $L$ and~\eqref{eqn:Tktangential_mixed} follows immediately from the definitions of the Newton tensors.  To prove~\eqref{eqn:Tktangentialdivergence}, first recall that the Newton tensors are divergence-free with respect to the flat metric in $X$.  From the definition of the second fundamental form, we have that
 \[ w_{\alpha\beta,\gamma} = w_{\alpha\beta,\bar\gamma} + L_{\alpha\gamma}w_{\beta n} + L_{\beta\gamma}w_{\alpha n} . \]
 Inserting this into the definition of the Netwon tensors yields the result (cf.\ \cite{Chen2009}*{Lemma~11}).
\end{proof}

Lemma~\ref{lem:boundary_derivatives} allows us to carefully perform the integration by parts argument as described above.

\begin{proof}[Proof of Proposition~\ref{prop:main_prop}]
 Denote $\mC:=C^1(\oX)\cap C^2(M)$.  Define
 \begin{align*}
  \tilde U_1 & := -k\sum_p \oint_M uw_n^p\sigma_{k-1}(D^2w\rv_{TM}^{\w p}) d\mu, \\
  \hat U_1 & := -k\sum_p \oint_M uw_\alpha^p T_{k-1}(D^2w)_{\alpha n} d\mu .
 \end{align*}
 It follows from~\eqref{eqn:tangential_hessian} that $\tilde U_1$ is well-defined on $\mC$; i.e.\ $\tilde U_1$ depends on at most second-order tangential derivatives and first-order transverse derivatives of $w^1,\dotsc,w^k$ on $M$.  Furthermore, we have that
 \[ U_1 = \tilde U_1 + \hat U_1 . \]

 Consider now $\hat U_1+U_2+U_4$.  Define
 \begin{align*}
  W_1 & := -\frac{k}{k-1}\sum_{p\neq q}\oint_M uw_\alpha^p T_{k-2}(D^2w\rv_{TM}^{\w p,q})_{\alpha\beta}L_{\beta\gamma}w_\gamma^q  d\mu, \\
  W_2 & := -\frac{k}{k-1}\sum_{p\neq q}\oint_M uw_n^p T_{k-2}(D^2w\rv_{TM}^{\w p,q})_{\alpha\beta}w_{\bar\alpha\bar\beta}^q d\mu, \\
  W_3 & := \frac{k}{k-1}\sum_{p\neq q\neq r}\oint_M uw_n^pw_\alpha^qT_{k-2}(L,D^2w\rv_{TM}^{\w p,q,r})_{\alpha\beta}L_{\beta\gamma}w_\gamma^r d\mu.
 \end{align*}
 It follows from~\eqref{eqn:tangential_hessian} that $W_1,W_2,W_3$ are well-defined on $\mC$.  Define also
 \begin{align*}
  V_1 & := \frac{k}{k-1}\sum_{p\neq q}\oint_M w^pu_\alpha T_{k-2}(D^2w\rv_{TM}^{\w p,q})_{\alpha\beta}L_{\beta\gamma}w_\gamma^qd\mu, \\
  V_2 & := -\frac{k}{k-1}\sum_{p\neq q\neq r} \oint_M w^pw_n^qu_\alpha T_{k-2}(L,D^2w\rv_{TM}^{\w p,q,r})_{\alpha\beta} L_{\beta\gamma} w_\gamma^rd\mu,
 \end{align*}
 Note that $V_1$ and $V_2$ still involve derivatives of $u$; this issue will be dealt with later.  Integrating by parts along $M$ and using Lemma~\ref{lem:boundary_derivatives} yields
 \begin{align*}
  \hat U_1 + U_2 + U_4 & = W_1 + V_1 + \frac{k}{k-1}\sum_{p\neq q}\oint_M uw_\alpha^p T_{k-2}(D^2w\rv_{TM}^{\w p,q})_{\alpha\beta}w_{n\bar\beta}^q d\mu\\
   & \quad - \frac{k}{k-1}\sum_{p\neq q}\oint_M w^p T_{k-2}(D^2w\rv_{TM}^{\w p,q})_{\alpha\beta}(u_\alpha w_n^q)_{\bar\beta} d\mu\\
  & = W_1 + V_1 + V_2 + \frac{k}{k-1}\sum_{p\neq q}\oint_M w_\alpha^p T_{k-2}(D^2w\rv_{TM}^{\w p,q})_{\alpha\beta}(uw_n)_{\bar\beta}d\mu \\
   & \quad + \frac{k}{k-1}\sum_{p\neq q\neq r} \oint_M w^pw_n^qu_\alpha T_{k-2}(L,D^2w\rv_{TM}^{\w p,q,r})_{\alpha\beta}w_{n\bar\beta}^rd\mu \\
  & = W_1 + W_2 + W_3 + V_1 + V_2 + \hat U_2 + \hat U_3 ,
 \end{align*}
 where
 \begin{align*}
  \hat U_2 & := -\frac{k}{k-1}\sum_{p\neq q\neq r} \oint_M uw_n^pw_\alpha^q T_{k-2}(L,D^2w\rv_{TM}^{\w p,q,r})_{\alpha\beta}w_{n\bar\beta}^rd\mu, \\
  \hat U_3 & := \frac{k}{k-1}\sum_{p\neq q\neq r} \oint_M w^pw_n^qu_\alpha T_{k-2}(L,D^2w\rv_{TM}^{\w p,q,r})_{\alpha\beta}w_{n\bar\beta}^rd\mu .
 \end{align*}

 We continue this process by considering $\hat U_2+\hat U_3+U_6$.  More generally, given $1\leq i\leq k-1$, we make the following definitions.  First, define
 \begin{align*}
  W_{2i-1} & := (-1)^i\frac{k}{k-1}\sum_{p_0\neq\dotsb\neq p_i} \frac{1}{(i-1)!}\oint_M uw_\alpha^{p_0}w_n^{p_1}\dotsm w_n^{p_{i-1}}
 \\ & \quad \quad \quad  \quad \quad \quad \quad \quad \quad \quad \quad \quad \quad \times
T_{k-2}(\overbrace{L,\dotsc,L}^{i-1},D^2w\rv_{TM}^{\w p_0,\dotsc,p_i})_{\alpha\beta}L_{\beta\gamma}w_\gamma^{p_{i}}d\mu, \\
  W_{2i} & := (-1)^i\frac{k}{k-1}\sum_{p_0\neq\dotsb\neq p_{i}}\frac{1}{i!}\oint_M uw_n^{p_0}\dotsm w_n^{p_{i-1}} \\
   & \quad\quad\quad\quad\quad\quad\quad\quad\quad\quad\quad\quad\quad\times T_{k-2}(\overbrace{L,\dotsc,L}^{i-1},D^2w\rv_{TM}^{\w p_0,\dotsc,p_{i}})_{\alpha\beta}w_{\bar\alpha\bar\beta}^{p_{i}} d\mu.
 \end{align*}
 It follows from~\eqref{eqn:tangential_hessian} that $W_{2i-1}$ and $W_{2i}$ are well-defined on $\mC$.  Next, define
 \begin{align*} V_i &:= (-1)^{i+1}\frac{k}{k-1}\sum_{p_0\neq\dotsb\neq p_{i}} \frac{1}{(i-1)!}\oint_M u_\alpha w^{p_0}w_n^{p_1}\dotsm w_n^{p_{i-1}} \\ &\quad \quad \quad \quad \quad \quad  \quad \quad \quad \quad \quad \quad \quad \quad \quad \times
T_{k-2}(\overbrace{L,\dotsc,L}^{i-1},D^2w\rv_{TM}^{\w p_0,\dotsc,p_{i}})_{\alpha\beta}L_{\beta\gamma}w_\gamma^{p_{i}}d\mu .\\ \end{align*}
 Note that $V_i$ still involves derivatives of $u$; this issue will be dealt with later.  Finally, define
 \begin{align*}
  \hat U_{2i} & := (-1)^i\frac{k}{k-1}\sum_{p_0\neq\dotsb\neq p_{i+1}} \frac{1}{i!}\oint_M uw_\alpha^{p_0}w_n^{p_1}\dotsm w_n^{p_i}
 \\ & \quad \quad \quad  \quad \quad \quad \quad \quad \quad \quad \quad \quad \quad \times T_{k-2}(\overbrace{L,\dotsc,L}^i,D^2w\rv_{TM}^{\w p_0,\dotsc,p_{i+1}})_{\alpha\beta}w_{n\bar\beta}^{p_{i+1}}d\mu, \\
  \hat U_{2i+1} & := (-1)^{i+1}\frac{k}{k-1}\sum_{p_0\neq\dotsb\neq p_{i+1}} \frac{1}{i!}\oint_M u_\alpha w^{p_0}w_n^{p_1}\dotsm w_n^{p_i}
 \\
& \quad \quad \quad  \quad \quad \quad \quad \quad \quad \quad \quad \quad \quad \times T_{k-2}(\overbrace{L,\dotsc,L}^i,D^2w\rv_{TM}^{\w p_0,\dotsc,p_{i+1}})_{\alpha\beta}w_{n\bar\beta}^{p_{i+1}}d\mu ;
 \end{align*}
 note that $\hat U_{2k-2}=\hat U_{2k-1}=0$.  Integrating by parts along $M$ and using Lemma~\ref{lem:boundary_derivatives} yields
 \[ \hat U_{2i} + \hat U_{2i+1} + U_{2i+4} = V_{i+2} + W_{2i+2} + W_{2i+3} + \hat U_{2i+2} + \hat U_{2i+3} . \]
 In particular, it follows that
 \begin{equation}
  \label{eqn:Usimplification}
  \sum_{i=1}^{2k+1} U_i = \tilde U_1 + \sum_{i=1}^k U_{2i+1} + \sum_{i=1}^{2k-2}W_i + \sum_{i=1}^{k-1} V_i .
 \end{equation}
 Note that $\tilde U_1$, $\sum U_{2i+1}$, and $\sum W_i$ are all well-defined on $\mC$.  It remains to check that, after integration by parts, $\sum V_i$ can be written as a boundary integral with integrand the product of $u$ with a function which is well-defined on $\mC$.

 Given $1\leq i\leq k-1$, define
 \begin{align*}
  A_i & := (-1)^i\frac{k}{(i-1)!(k-1)}\sum_{p_0\neq\dotsb\neq p_{i+1}} \oint_M uw^{p_0}w_n^{p_1}\dotsm w_n^{p_{i-1}} w_\alpha^{p_i} w_{n\bar\beta}^{p_{i+1}} \\
& \quad \quad \quad \quad \quad  \quad \quad \quad \quad \quad \quad \quad \quad \quad \quad \times T_{k-2}(\overbrace{L,\dotsc,L}^i,D^2w\rv_{TM}^{\w p_0,\dotsc,p_{i+1}})_{\alpha\gamma}L_{\gamma\beta} d\mu, \\
  B_i & := (-1)^{i+1}\frac{k}{(i-1)!(k-1)}\sum_{p_0\neq\dotsb\neq p_{i+1}} \oint_M uw^{p_0}w_n^{p_1}\dotsm w_n^{p_{i-1}} w_\gamma^{p_i} w_\delta^{p_{i+1}} 
 \\
& \quad \quad \quad  \quad \quad \quad \quad \quad \quad \quad \quad \quad \quad \times T_{k-2}(\overbrace{L,\dotsc,L}^i,D^2w\rv_{TM}^{\w p_0,\dotsc,p_{i+1}})_{\alpha\beta} L_{\alpha\gamma} L_{\beta\delta} d\mu, \\
  C_i & := (-1)^i\frac{k}{(i-1)!(k-1)}\sum_{p_0\neq\dotsb\neq p_i} \oint_M uw_n^{p_0}\dotsm w_n^{p_{i-2}}
 \\
&\quad \quad \quad \quad \quad  \quad  \quad \quad \quad \quad \times T_{k-2}(\overbrace{L,\dotsc,L}^{i-1},D^2w\rv_{TM}^{\w p_0,\dotsc,p_i})_{\alpha\beta}\left(w^{p_{i-1}}w_\gamma^{p_i}L_{\alpha\gamma}\right)_{\bar\beta} d\mu.
 \end{align*}
 Note that $B_i$ and $C_i$ are well-defined on $\mC$.  Moreover, integration by parts along $M$ readily yields
 \[ V_i = A_i - A_{i-1} + B_i + C_i, \]
 where we interpret $A_0=0$.  Since $A_{k-1}=0$, it follows that
 \begin{equation}
  \label{eqn:Vsimplification}
  \sum_{i=1}^{k-1} V_i = \sum_{i=1}^{k-1} \left( B_i + C_i\right) .
 \end{equation}
 Combining~\eqref{eqn:Usimplification} and~\eqref{eqn:Vsimplification} yields the desired result.
\end{proof}

\section{The first and second variation}
\label{sec:variation}

It is straightforward to compute the first and second variations of the energy functional
\[ \mE_k(u) := \mQ_k(u,\dotsc,u) \]
associated to the symmetric multilinear form constructed by Proposition~\ref{prop:main_prop}.

\begin{proposition}
 \label{prop:first_variation}
 Let $X\subset\bR^n$ be a bounded smooth domain with boundary $M=\partial X$.  Let $u,v\in C^\infty(\oX)$ and suppose that $v\rv_M=0$.  Then
 \begin{equation}
  \label{eqn:first_variation}
  \left.\frac{d}{dt}\right|_{t=0}\mE_k(u+tv) = -(k+1)\int_X v\,\sigma_k(D^2u,\dotsc,D^2u) dx.
 \end{equation}
\end{proposition}

\begin{proof}
 Since $\mQ_k$ is symmetric, we compute that
 \[ \left.\frac{d}{dt}\right|_{t=0}\mE_k(u+tv) = (k+1)\mQ_k(v,u,\dotsc,u) . \]
 Since $v\rv_M=0$, we see that the boundary integral in~\eqref{eqn:k-energy} vanishes.  This yields~\eqref{eqn:first_variation}.
\end{proof}

\begin{proposition}
 \label{prop:second_variation}
 Let $X\subset\bR^n$ be a bounded smooth domain with boundary $M=\partial X$.  Let $u,v\in C^\infty(\oX)$ and suppose that $v\rv_M=0$.  Then
 \[ \left.\frac{d^2}{dt^2}\right|_{t=0}\mE_k(u+tv) = (k+1)\int_X v_iv_j T_{k-1}(D^2u)_{ij}dx . \]
 In particular, if $u\in\overline{\Gamma_k^+}$, then
 \[ \left.\frac{d^2}{dt^2}\right|_{t=0}\mE_k(u+tv)\geq0 \]
 for all $v\in C^\infty(\oX)$ such that $v\rv_M=0$.
\end{proposition}

\begin{proof}
 Since $\mQ_k$ is symmetric, we compute that
 \[ \left.\frac{d^2}{dt^2}\right|_{t=0}\mE_k(u+tv) = k(k+1)\mQ_k(v,v,u,\dotsc,u) . \]
 Since $v\rv_M=0$, it follows that
 \begin{align*}
  \left.\frac{d^2}{dt^2}\right|_{t=0}\mE_k(u+tv) & = -k(k+1)\int_X v\,\sigma_k(D^2v,D^2u,\dotsc,D^2u) dx, \\
  & = -(k+1)\int_X v T_{k-1}(D^2u)_{ij} v_{ij}dx \\
  & = (k+1)\int_X v_iv_j T_{k-1}(D^2u)_{ij} dx.
 \end{align*}
 The last conclusion follows from the fact that if $u\in\overline{\Gamma_k^+}$, then $T_{k-1}(D^2u)_{ij}$ is nonnegative.
\end{proof}

We are now ready to prove Theorem~\ref{thm:main_thm}, which we restate here for convenience.

\begin{theorem}
 Let $X\subset\bR^n$ be a bounded smooth domain with $(k-1)$-convex boundary $M=\partial X$.  Fix $f\in C^\infty(M)$ and denote
 \[ \mC_{f,k} = \left\{ u\in \Gamma_k^+\suchthat u\rv_M = f \right\} . \]
 Then
 \[ \mE_k(u) \geq \mE_k(u_f) \]
 for all $u\in\overline{\mC_{f,k}}$, where $u_f\in\overline{\mC_{f,k}}$ is the solution to the Dirichlet problem
 \begin{equation}
  \label{eqn:dirichlet_problem_proof}
  \begin{cases}
    \sigma_k(u_f) = 0, & \text{in $X$}, \\
    u_f = f, & \text{on $M$}.
   \end{cases}
 \end{equation}
\end{theorem}

\begin{proof}
 By Proposition~\ref{prop:first_variation}, the solution $u_f$ to~\eqref{eqn:dirichlet_problem_proof} is a critical point of the functional $\mE_k\colon C^{1,1}(\oX)\to\bR$.  By Proposition~\ref{prop:second_variation}, the restriction $\mE_k\colon\overline{\mC_{f,k}}\to\bR$ is a convex functional.  Since $\overline{\mC_{f,k}}$ is convex, $u_f$ realizes the infimum of $\mE_k\colon\overline{\mC_{f,k}}\to\bR$.  Indeed, if not, then there is a $u\in\overline{\mC_{f,k}}$ such that $\mE_k(u)<\mE_k(u_f)$.  Since $\overline{\mC_{f,k}}$ is convex, it follows that $tu+(1-t)u_f\in\overline{\mC_{f,k}}$ for all $t\in[0,1]$.  Denote $\mE_k(t):=\mE_k(tu+(1-t)u_f)$.  Since $\mE_k(u)<\mE_k(u_f)$, there exists a $t^\ast\in[0,1]$ such that $\mE_k^\prime(t^\ast)<0$.  This contradicts the facts that $\mE_k^\prime(0)=0$ and $\mE_k^{\prime\prime}\geq0$ for all $t\in[0,1]$.
\end{proof}

\section{The case $k=2$}
\label{sec:example}

We conclude this article by considering the specific case $k=2$; the case $k=1$ is covered by~\eqref{eqn:model_trace}.  First, a suitable boundary operator as in Proposition~\ref{prop:main_prop} is given by Proposition~\ref{prop:main_prop2}, which we restate here for convenience.

\begin{proposition}
 \label{prop:boundary2}
 Let $X\subset\bR^n$ be a bounded smooth domain with boundary $M=\partial X$.  Define $B\colon\left(C^1(\oX)\cap C^2(M)\right)^2\to C^0(M)$ by
 \begin{equation}
  \label{eqn:boundary2}
  B_2(v,w) = \frac{1}{2}\left(v_n\oDelta w + w_n\oDelta v + L(\onabla v,\onabla w) + Hv_nw_n\right) .
 \end{equation}
 Then the multilinear form $\mQ_2\colon\left(C^2(\oX)\right)^3\to\bR$ given by
 \[ \mQ_2(u,v,w) = -\int_X u\sigma_2(D^2v,D^2w) dx+ \oint_M u B_2(v,w)d\mu \]
 is symmetric.
\end{proposition}

\begin{proof}
 Following the proof of Proposition~\ref{prop:main_prop}, we see that a suitable choice of boundary operator is
 \begin{multline*}
  \tilde B_2(v,w) := \frac{1}{2}\left(v_n\oDelta w + w_n\oDelta v + L(\onabla v,\onabla w) + Hv_nw_n\right) \\ + \frac{1}{6}\left( A(\onabla v,\onabla w) + v\lp A,\bar{D}^2w\rp + w\lp A,\bar{D}^2v\rp + v\lp\onabla H,\onabla w\rp + w\lp\onabla H,\onabla v\rp \right) .
 \end{multline*}
 A straightforward computation yields
 \begin{multline*}
  \odelta\left(vA(\onabla w)\right) + \odelta\left(wA(\onabla v)\right) - A(\onabla v,\onabla w) \\ = A(\onabla v,\onabla w) + v\lp A,\bar{D}^2w\rp + w\lp A,\bar{D}^2v\rp + v\lp\onabla H,\onabla w\rp + w\lp\onabla H,\onabla v\rp
 \end{multline*}
 On the other hand,
 \begin{multline*}
  \oint_M u\left[\,\odelta\left(vA(\onabla w)\right) + \odelta\left(wA(\onabla v)\right) - A(\onabla v,\onabla w)\right]d\mu \\ = -\oint_M \left[ uA(\onabla v,\onabla w) + vA(\onabla w,\onabla u) + wA(\onabla u,\onabla v) \right]d\mu
 \end{multline*}
 is symmetric in $u,v,w$.  Thus $B_2-\tilde B_2$, and hence $\mQ_2$, is symmetric in $u,v,w$. 
\end{proof}

Applying this boundary operator in Theorem~\ref{thm:main_thm} yields the following sharp Sobolev trace inequality.

\begin{theorem}
 \label{thm:trace2}
 Let $X\subset\bR^n$ be a bounded smooth mean-convex domain with boundary $M=\partial X$.  Given $f\in C^\infty(M)$, set
 \[ \mC_f = \left\{ u\in \Gamma_2^+ \suchthat u\rv_M=f \right\} . \]
 Then it holds that
 \[ -\int_X u\sigma_2(D^2u) dx+ \oint_M uB_2(u,u) d\mu\geq \oint_M fB_2(u_f,u_f)d\mu \]
 for all $u\in\overline{\mC_f}$, where $B_2$ is the operator~\eqref{eqn:boundary2} and $u_f\in C^{1,1}(\oX)\cap\overline{\Gamma_2^+}$ is the unique solution to the Dirichlet problem
 \[ \begin{cases}
     \sigma_2(D^2u_f)=0, & \text{in $X$}, \\
     u = f, & \text{on $M$} .
    \end{cases} \]
\end{theorem}

\end{document}